\documentclass[11pt]{amsart}

\usepackage{amssymb,amsmath,color}

\theoremstyle{plain}
\newtheorem{thm}{Theorem}%[section]
\newtheorem{lem}{Lemma}

\newtheorem{prop}{Proposition}

\theoremstyle{definition}
\newtheorem{defn}{Definition}

\allowdisplaybreaks

\begin{document}

\title[On asymptotic expansions of generalized Bergman kernels]{On asymptotic expansions of generalized Bergman kernels on symplectic manifolds}

\author[Y. A. Kordyukov]{Yuri A. Kordyukov}
\address{Institute of Mathematics with Computing Centre\\
Ufa Federal Research Centre of 
         Russian Academy of Sciences\\
         112~Chernyshevsky str.\\ 450008 Ufa\\ Russia} \email{yurikor@matem.anrb.ru}

\thanks{The research is supported by the grant of Russian Science Foundation (project no. 17-11-01004)}



\begin{abstract}
A full off-diagonal asymptotic expansion is established for the generalized
Bergman kernels of the renormalized Bochner Laplacians associated with high tensor powers of a positive line bundle over a compact symplectic manifold. As an application, the algebra of Toeplitz operators on the symplectic manifold associated with the renormalized Bochner Laplacian is constructed. 
\end{abstract}

\dedicatory{Dedicated to the 130th anniversary of Vladimir Ivanovich Smirnov's birth}
\date{}

 \maketitle
%\tableofcontents
\section{Introduction}\label{s:intro}
In this paper, we study the asymptotic behavior of the generalized Bergman kernels of the renormalized Bochner-Laplacians associated to high tensor powers of a positive line bundle over a compact symplectic manifold. So we consider a compact symplectic manifold  $(X,\omega)$ of dimension $2n$. Let $(L,h^L)$ be a Hermitian line bundle on $X$ with a Hermitian connection $\nabla^L : C^\infty(X,L)\to C^\infty(X,T^*X\otimes L)$. The curvature of this connection is given by $R^L=(\nabla^L)^2$. 
We will assume that $L$ satisfies the prequantization condition:
\[
\frac{i}{2\pi}R^L=\omega. 
\] 
Thus, $[\omega]\in H^2(X,\mathbb Z)$. Let $(E, h^E)$ be a Hermitian vector bundle on $X$ with Hermitian connection $\nabla^E$, and let $R^E$ be the curvature of $\nabla^E$.

Let $g$ be a Riemannian metric on $X$. Let $J_0 : TX\to TX$ be a skew-symmetric operator such that 
\[
\omega(u,v)=g(J_0u,v), \quad u,v\in TX. 
\]
Consider the operator $J : TX\to TX$ given by 
$J=J_0(-J^2_0)^{-1/2}$. 
Then $J$ is an almost complex structure compatible with $\omega$ and $g$, that is, $g(Ju,Jv)=g(u,v), \omega(Ju,Jv)=\omega(u,v)$ and $\omega(u,Ju)\geq 0$ for any $u,v\in TX$.  
  
Let $\nabla^{TX}$ be the Levi-Chivita connection of the metric $g$. For any $p\in \mathbb N$ denote by $L^p$ the $p$th tensor power of $L$. Let $\nabla^{L^p\otimes E}: C^\infty(X,L^p\otimes E)\to C^\infty(X,L^p\otimes E\otimes T^*X)$ be the connection on $L^p\otimes E$ induced by $\nabla^{L}$ and $\nabla^E$.
Denote by $\Delta^{L^p\otimes E}$ the induced Bochner-Laplacian acting on $C^\infty(X,L^p\otimes E)$ by the formula 
\[
\Delta^{L^p\otimes E}=\left(\nabla^{L^p\otimes E}\right)^*\nabla^{L^p\otimes E},
\]
where $\left(\nabla^{L^p\otimes E}\right)^*: C^\infty(X,L^p\otimes E\otimes T^*X)\to C^\infty(X,L^p\otimes E)$ stands for the formal adjoint of the operator $\nabla^{L^p\otimes E}$, and by $\Delta_p$ the renormalized Bochner Laplacian acting on $C^\infty(X,L^p\otimes E)$ by the formula
 \[
\Delta_p=\Delta^{L^p\otimes E}-p\tau,
 \]
where $\tau\in C^\infty(X)$ is given by
 \[
 \tau(x)=-\pi \operatorname{Tr}[J_0(x)J(x)],\quad x\in X.
 \]
The renormalized Bochner-Laplacian $\Delta_p$ was introduced by V. Guillemin and A. Uribe in \cite{Gu-Uribe}. When $(X,\omega)$ is a Kaehler manifold, it is twice the corresponding Kodaira-Laplacian on functions $\Box^{L^p}=\bar\partial^{L^p*}\bar\partial^{L^p}$.
 The asymptotics of its spectrum as $p\to \infty$ was studied in \cite{B-Uribe,Brav,Gu-Uribe,ma-ma02,ma-ma08}.    
 
Denote by $\sigma(\Delta_p)$ the spectrum of $\Delta_p$ in $L^2(X,L^p\otimes E)$. By \cite[Corollary 1.2]{ma-ma08} (see also \cite{Bismut-V,B-Uribe,Brav,Gu-Uribe}), there exists a constant $C_L>0$ such that 
 \[
 \sigma(\Delta_p)\subset [-C_L,C_L]\cup [2p\mu_0-C_L,+\infty),
 \]
for any $p$, where $\mu_0>0$ is given by
 \[
 \mu_0=\inf_{u\in T_xX, x\in X}\frac{iR^L_x(u,J(x)u)}{|u|_g^2}.
 \]
Consider the linear subspace $\mathcal H_p\subset L^2(X,L^p\otimes E)$ spanned by the eigensections of $\Delta_p$ corresponding to eigenvalues in $[-C_L,C_L]$. Let $P_{\mathcal H_p}$ be the orthogonal projection in $L^2(X,L^p\otimes E)$ onto $\mathcal H_p$. The smooth kernel $P_{q,p}(x,x^\prime)$, $x,x^\prime\in X$, of the operator $(\Delta_p)^qP_{\mathcal H_p}$ with respect to the Riemannian volume form $dv_X$ is called a generalized Bergman kernel of $\Delta_p$.

We are interested in the asymptotic behavior of the generalized Bergman kernel $P_{q,p}(x,x^\prime)$ as $p\to \infty$. 
First, we recall that, by \cite{ma-ma08}, for any $m\in \mathbb N$ and $\varepsilon>0$, we have
\[
|P_{q,p}(x,x^\prime)|_{C^m}=\mathcal O(p^{-\infty})
\] 
if $d(x,x^\prime)>\varepsilon$. To describe the asymptotic expansion of $P_{q,p}(x,x^\prime)$ near the diagonal, we introduce normal coordinates near an arbitrary point $x_0\in X$. 

Let $a^X$ be the injectivity radius of the Riemannian manifold $(X,g)$. We denote by $B^{X}(x_0,r)$ and $B^{T_{x_0}X}(0,r)$ the open balls in $X$ and $T_{x_0}X$ with center $x_0$ and radius $r$, respectively. We identify $B^{T_{x_0}X}(0,a^X)$ with $B^{X}(x_0,a^X)$ via the exponential map $\exp^X_{x_0}$. Furthermore, we choose trivializations of the bundles $L$ and $E$ over $B^{X}(x_0,a^X)$,   identifying their fibers $L_Z$ and $E_Z$ at $Z\in B^{T_{x_0}X}(0,a^X)\cong B^{X}(x_0,a^X)$ with the spaces  $L_{x_0}$ and $E_{x_0}$ by parallel transport with respect to the connections $\nabla^L$ and $\nabla^E$ along the curve $\gamma_Z : [0,1]\ni u \to \exp^X_{x_0}(uZ)$. Denote by $\nabla^{L^p\otimes E}$ and $h^{L^p\otimes E}$ the connection and the Hermitian metric on the trivial bundle with fiber $(L^p\otimes E)_{x_0}$ induced by these trivializations.  

Let $dv_{TX}$ denote the Riemannian volume form of the Euclidean space $(T_{x_0}X, g_{x_0})$. We define a smooth function $\kappa$ on 
\[
B^{T_{x_0}X}(0,a^X)\cong B^{X}(x_0,a^X)
\]
by the equation
\[
dv_{X}(Z)=\kappa(Z)dv_{TX}(Z), \quad Z\in B^{T_{x_0}X}(0,a^X). 
\] 

The almost complex structure $J_{x_0}$ induces a decomposition 
\[
T_{x_0}X\otimes_{\mathbb R}\mathbb C=T^{(1,0)}_{x_0}X\oplus T^{(0,1)}_{x_0}X,
\] 
where $T^{(1,0)}_{x_0}X$ and $T^{(0,1)}_{x_0}X$ are the eigenspaces of $J_{x_0}$ corresponding to eigenvalues $i$ and $-i$ respectively. Denote by $\det_{\mathbb C}$ the determinant function of the complex space $T^{(1,0)}_{x_0}X$. Put
\[
\mathcal J_{x_0}=-2\pi i J_0.
\]
Then $\mathcal J_{x_0} : T^{(1,0)}_{x_0}X\to T^{(1,0)}_{x_0}X$ is positive, and $\mathcal J_{x_0} : T_{x_0}X\to T_{x_0}X$ is skew-symmetric. We define a function $\mathcal P=\mathcal P_{x_0}\in C^\infty(T_{x_0}X\times T_{x_0}X)$ by
\begin{multline}\label{e:Bergman}
\mathcal P(Z, Z^\prime)\\=\frac{\det_{\mathbb C}\mathcal J_{x_0}}{(2\pi)^n}\exp\left(-\frac 14\langle (\mathcal J^2_{x_0})^{1/2}(Z-Z^\prime), (Z-Z^\prime)\rangle +\frac 12 \langle \mathcal J_{x_0} Z, Z^\prime \rangle \right).
\end{multline}
It is the Bergman kernel of the second order differential operator $\mathcal L_0$ on $C^\infty(T_{x_0}X,E_{x_0})$ given by
\begin{equation}\label{e:L0}
\mathcal L_0=-\sum_{j=1}^{2n} \left(\nabla_{e_j}+\frac 12 R^L_{x_0}(Z,e_j)\right)^2-\tau (x_0), 
\end{equation}
where $\{e_j\}_{j=1,\ldots,2n}$ is an orthonormal base in $T_{x_0}X$. 
Here, for $U\in T_{x_0}X$, we denote by $\nabla_U$ the ordinary operator of differentiation in the direction $U$ on the space $C^\infty(T_{x_0}X,E_{x_0})$. Thus, $\mathcal P$ is the smooth kernel (with respect to $dv_{TX}$) of the orthogonal projection in $L^2(T_{x_0}X,E_{x_0})$ to the kernel $N$ of $\mathcal L_0$.

Consider the fiberwise product 
\[
TX\times_X TX=\{(Z,Z^\prime)\in T_{x_0}X\times T_{x_0}X : x_0\in X\}.
\] 
Let $\pi : TX\times_X TX\to X$ be the natural projection given by  $\pi(Z,Z^\prime)=x_0$. The kernel $P_{q,p}(x,x^\prime)$ induces a smooth section $P_{q,p,x_0}(Z,Z^\prime)$ of the vector bundle $\pi^*(\operatorname{End}(E))$ on $TX\times_X TX$ defined for all $x_0\in X$ and $Z,Z^\prime\in T_{x_0}X$ with $|Z|, |Z^\prime|<a_X$. 

The main result of the paper is the following theorem, which states the existence of the full off-diagonal asymptotic expansion of the generalized Bergman kernel $P_{q,p}$ as $p\to \infty$. 

\begin{thm}\label{t:main}
There exists $\varepsilon\in (0,a^X)$ such that, for any $j,m,m^\prime\in \mathbb N$, $j\geq 2q$, there exist positive constants $C$, $c$ and $M$ such that for any $p\geq 1$, $x_0\in X$ and $Z,Z^\prime\in T_{x_0}X$, $|Z|, |Z^\prime|<\varepsilon$, we have
\begin{multline}
\sup_{|\alpha|+|\alpha^\prime|\leq m}\Bigg|\frac{\partial^{|\alpha|+|\alpha^\prime|}}{\partial Z^\alpha\partial Z^{\prime\alpha^\prime}}\Bigg(\frac{1}{p^n}P_{q,p,x_0}(Z,Z^\prime)\\ 
-\sum_{r=2q}^jF_{q,r,x_0}(\sqrt{p} Z, \sqrt{p}Z^\prime)\kappa^{-\frac 12}(Z)\kappa^{-\frac 12}(Z^\prime)p^{-\frac{r}{2}+q}\Bigg)\Bigg|_{\mathcal C^{m^\prime}(X)}\\ 
\leq Cp^{-\frac{j-m+1}{2}+q}(1+\sqrt{p}|Z|+\sqrt{p}|Z^\prime|)^M\exp(-c\sqrt{\mu_0p}|Z-Z^\prime|)+\mathcal O(p^{-\infty}),
\end{multline}
where
\[
F_{q,r,x_0}(Z,Z^\prime)=J_{q,r,x_0}(Z,Z^\prime)\mathcal P_{x_0}(Z,Z^\prime),
\]
the $J_{q,r,x_0}(Z,Z^\prime)$ are polynomials in $Z, Z^\prime$, depending smoothly on $x_0$, with the same parity as $r$ and $\operatorname{deg} J_{q,r,x_0}\leq 3r$.
\end{thm}

Here $\mathcal C^{m^\prime}(X)$ is the $\mathcal C^{m^\prime}$-norm for the parameter $x_0\in X$. We say that $G_p=\mathcal O(p^{-\infty})$ if for any $l, l_1\in \mathbb N$, there exists $C_{l,l_1}>0$ such that $\mathcal C^{l_1}$-norm of $G_p$ is estimated from above by $C_{l,l_1}p^{-l}$. 

The full off-diagonal asymptotic expansion of the Bergman kernel of the spin$^c$ Dirac operator associated to a positive line bundle on a compact symplectic manifold was proved by X. Dai, K. Liu and X. Ma \cite[Theorem 4.18']{dai-liu-ma} (see also \cite[Theorem 4.2.1]{ma-ma:book}). Their approach is inspired by the local index theory, especially by the analytic techniques of Bismut and Lebeau. In that case, it is very important that the eigenvalues of the associated Laplacian are either $0$ or tend to $+\infty$. 

In the current situation, the renormalized Bochner-Laplacian possibly have different bounded eigenvalues. Nevertheless, in \cite{ma-ma08}, X. Ma and G. Marinescu  developed the method to obtain a weaker result, a near diagonal asymptotic expansion of the generalized Bergman kernels of the renormalized Bochner Laplacian \cite[Theorem 1.19]{ma-ma08} (see also \cite[Theorem 4.1.24]{ma-ma:book}), which turned out to be sufficient for many applications. The paper \cite{ma-ma08} also contains computations of some coefficients $F_{q,r,x_0}$.  

In this paper, we modify the technique of Ma and Marinescu to prove the full off-diagonal asymptotic expansion of the generalized Bergman kernels of the renormalized Bochner Laplacian. We follow the strategy of \cite{dai-liu-ma,ma-ma08}. So the first step is localization of the problem in a neighborhood of the diagonal. Then we rescale the operator in the normal coordinates introduced above and obtain its formal asymptotic expansion as $p\to \infty$. Finally, we use the Riesz formula, the formal power series technique, Sobolev norm estimates and Sobolev embedding theorems. The most essential improvement which allows us to extend the domain of validity of asymptotic expansions is the use of weighted Sobolev spaces and weighted estimates. Here we apply the technique that was used earlier in \cite{Kor91,Kor00,Meladze-Shubin1,Meladze-Shubin2} to prove pointwise estimates for the kernels of functions of elliptic differential operators on noncompact manifolds. 

As an immediate application of Theorem~\ref{t:main}, we construct the algebra of Toeplitz operators on the symplectic manifold $X$ associated with the renormalized Bochner Laplacian. Actually, once we prove the full off-diagonal asymptotic expansion of the generalized Bergman kernels, we can easily get a characterization of Toeplitz operators and prove that these operators form an algebra, following the arguments of \cite{ma-ma08a}. In the process of preparation of this paper, X. Ma and G. Marinescu informed me about the preprint \cite{ioos-lu-ma-ma}, where they constructed the algebra of Toeplitz operators associated with the renormalized Bochner Laplacian, by using asymptotic expansions of the generalized Bergman kernels in two complimentary domains covering the manifold \cite[(2.5) and Theorem 2.1]{ioos-lu-ma-ma}  (see also \cite{lu-ma-ma}). These expansions are stronger than the near diagonal expansions proved in \cite{ma-ma08}, but weaker than the full off-diagonal ones established in the present paper. 

The paper is organized as follows. In \S\ref{local} we recall the results of \cite[Sections 1.1 and 1.2]{ma-ma08} on the localization and the rescaling of the problem which allow us to obtain a formal asymptotic expansion of the renormalized Bochner-Laplacian as $p\to \infty$. In \S\ref{norm}, we derive the weighted norm estimates for the resolvent of the renormalized Bochner-Laplacian $\mathcal L_t$. In \S\ref{norm-Bergman}, we derive the weighted norm estimates for the generalized Bergman projection $\mathcal P_{q,t}$ associated with the operator $\mathcal L_t$ and its derivatives of an arbitrary order with respect to $t$. In \S\ref{asymp}, we first use the estimates of Section~\ref{norm-Bergman} to derive the weighted norm estimates for the remainders in the asymptotic formula for the generalized Bergman projection $\mathcal P_{q,t}$. Then Sobolev embedding theorem allows us to obtain pointwise estimates of the remainders in the asymptotic formula for the generalized Bergman kernels. Finally, writing these pointwise estimates in the initial coordinates, we complete the proof of the main theorem. \S\ref{Toeplitz} is devoted to Toeplitz operators. 

The author is grateful to X. Ma and G. Marinescu for useful discussions. 

 \section{Localization and rescaling of the problem}\label{local}
In this section, we recall the results of \cite[Sections 1.1 and 1.2]{ma-ma08} on the localization and rescaling of the problem, which allow us to obtain a formal asymptotic expansion of the renormalized Bochner Laplacian as $p\to \infty$. 

We will keep the notation introduced in \S\ref{s:intro}. We fix $x_0\in X$. Let $\{e_j\}$ be an oriented orthonormal basis of $T_{x_0}X$. It gives rise to an isomorphism $X_0:=\mathbb R^{2n}\cong T_{x_0}X$.  Consider the trivial bundles $L_0$ and $E_0$ with fibers $L_{x_0}$ and $E_{x_0}$ on $X_0$. Recall that we have the Riemannian metric $g$ on $B^{T_{x_0}X}(0,a^X)$ as well as the connections $\nabla^L$, $\nabla^E$ and the Hermitian metrics $h^L$, $h^E$ on the restrictions of $L_0$ and $E_0$ to $B^{T_{x_0}X}(0,a^X)$ induced by the identification $B^{T_{x_0}X}(0,a^X)\cong B^{X}(x_0,a^X)$. In particular, $h^L$, $h^E$ are the constant metrics $h^{L_0}=h^{L_{x_0}}$, $h^{E_0}=h^{E_{x_0}}$. For $\varepsilon \in (0,a^X/4)$, one can extend these geometric objects from $B^{T_{x_0}X}(0,\varepsilon)$ to $\mathbb R^{2n}\cong T_{x_0}X$ in the following way.  

Let $\rho : \mathbb R\to [0,1]$ be a smooth even function such that $\rho(v)=1$ if $|v|<2$ and $\rho(v)=0$ if $|v|>4$. Let $\varphi_\varepsilon : \mathbb R^{2n}\to \mathbb R^{2n}$ be the map defined by $\varphi_\varepsilon(Z)=\rho(|Z|/\varepsilon)Z$. We equip $X_0$ with the metric $g_0(Z)=g(\varphi_\varepsilon(Z))$. Set $\nabla^{E_0}=\varphi^*_\varepsilon\nabla^E$. Define  the Hermitian connection $\nabla^{L_0}$ on $(L_0,h^{L_0})$ by 
\[
\nabla^{L_0}(Z)=\varphi_\varepsilon^*\nabla^L+\frac 12(1-\rho^2(|Z|/\varepsilon))R^L_{x_0}(\mathcal R,\cdot),
\]
where $\mathcal R(Z)=\sum_j Z_je_j\in \mathbb R^{2n}\cong T_ZX_0$. 

One can show that, for $\varepsilon$ small enough, the curvature $R^{L_0}$ is positive and satisfies the following estimate for any $x_0\in X$,
\[
\inf_{u\in TX}\frac{iR^{L_0}(u,J^{L_0}u)}{|u|_g^2}\geq (1-\alpha)\mu_0. 
\]
From now on, we fix such an $\varepsilon>0$. 

We also extend the function $\kappa$ from $B^{T_{x_0}X}(0,\varepsilon)$ to $X_0$. Let $dv_{X_0}$ be the Riemannian volume form of $(X_0, g_0)$. Then $\kappa$ is the smooth positive function on $X_0$ defined by the equation
\[
dv_{X_0}(Z)=\kappa(Z)dv_{TX}(Z), \quad Z\in X_0. 
\] 

Let $\Delta_p^{X_0}=\Delta^{L_0^p\otimes E_0}-p\tau_0$ be the associated renormalized Bochner Laplacian acting on $C^\infty(X_0,L_0^p\otimes E_0)$. Then (cf. \cite[Equation (1.23)]{ma-ma08}) there exists a constant $C_{L_0}>0$ such that for any $p$ we have
 \begin{equation}\label{gap0}
 \sigma(\Delta_p^{X_0})\subset [-C_{L_0},C_{L_0}]\cup [2(1-\alpha)p\mu_0-C_{L_0},+\infty).
 \end{equation}
 
Consider the subspace $\mathcal H^0_p$ in $C^\infty(X_0,L_0^p\otimes E_0)\cong C^\infty(\mathbb R^{2n}, E_{x_0})$ spanned by the eigensections of $\Delta^{X_0}_p$ corresponding to eigenvalues in $[-C_{L_0},C_{L_0}]$. Let $P_{\mathcal H^0_p}$ be the orthogonal projection onto $\mathcal H^0_p$. The smooth kernel of $(\Delta^{X_0}_p)^qP_{\mathcal H^0_p}$ with respect to the Riemannian volume form $dv_{X_0}$ is denoted by $P^0_{q,p}(Z,Z^\prime)$. The kernels $P_{q,p,x_0}(Z,Z^\prime)$ and $P^0_{q,p}(Z,Z^\prime)$ are asymptotically close on $B^{T_{x_0}X}(0,\varepsilon)$ in the $\mathcal C^\infty$-topology, as $p\to \infty$. 

\begin{prop}[\cite{ma-ma08}, Proposition 1.3]\label{prop13}
For any $l,m\in \mathbb N$, there exists $C_{l,m}>0$ such that for $x_0\in X$  and $x,x^\prime\in B^{T_{x_0}X}(0,\varepsilon)$, we have 
 \[
|P_{q,p,x_0}(Z,Z^\prime)-P^0_{q,p}(Z,Z^\prime)|_{\mathcal C^m}\leq C_{l,m}p^{-l}.  
\]
 \end{prop}
 
Next, we use the rescaling introduced in \cite[Section 1.2]{ma-ma08}. 
Denote $t=\frac{1}{\sqrt{p}}$. For $s\in C^\infty(\mathbb R^{2n}, E_{x_0})$, set
\[
S_ts(Z)=s(Z/t), \quad Z\in \mathbb R^{2n}.
\]
The rescaled connection $\nabla_t$ is defined as 
\[
\nabla_t=tS^{-1}_t\kappa^{\frac 12}\nabla^{L_0^p}\kappa^{-\frac 12}S_t.
\]
Let $\Gamma^L$, $\Gamma^E$ be the connection forms of $\nabla^L$, $\nabla^E$ with respect to some fixed frames for $L$ and $E$ which are parallel along the curves $\gamma_Z : [0,1]\ni u \to \exp^X_{x_0}(uZ)$ under the chosen trivializations on $B^{T_{x_0}X}(0,\varepsilon)$. Then on $B^{T_{x_0}X}(0,\varepsilon/t)$, we have
\begin{align*}
\nabla_{t,e_i}&=\kappa^{\frac 12}(tZ)\left(\nabla_{e_i}+\frac{1}{t}\Gamma^L(e_i)(tZ)+t\Gamma^E(e_i)(tZ)\right)\kappa^{-\frac 12}(tZ)\\
&=\nabla_{e_i}+\frac{1}{t}\Gamma^L(e_i)(tZ)+t\Gamma^E(e_i)(tZ)-t\left(\kappa^{-1}(e_i\kappa)\right)(tZ).
\end{align*}

Recall that 
\[
\sum_{|\alpha|=r}(\partial^\alpha\Gamma^L)_{x_0}(e_j)\frac{Z^\alpha}{\alpha!}=\frac{1}{r+1}\sum_{|\alpha|=r-1}(\partial^\alpha  R^L)_{x_0}(\mathcal R,e_j)\frac{Z^\alpha}{\alpha!}.
\]
In particular, 
\[
\Gamma^L(e_j)(Z)=\frac{1}{2}R^L_{x_0}(\mathcal R,e_j)+O(|Z|^2). 
\]
Similar identities are valid for $\Gamma^E$. 

The rescaled operator $\mathcal L_t$ is defined to be
\begin{equation}\label{scaling}
\mathcal L_t=t^2S^{-1}_t\kappa^{\frac 12}\Delta^{X_0}_p\kappa^{-\frac 12}S_t.
\end{equation}
We have
\[
\mathcal L_t=-g^{jk}(tZ)\left[\nabla_{t,e_j}\nabla_{t,e_k}- t\Gamma^{\ell}_{jk}(tZ)\nabla_{t,e_\ell}\right]-\tau(tZ).
\] 
By \eqref{gap0}, it follows that
 \[
 \sigma(\mathcal L_t)\subset [-C_{L_0}t^2,C_{L_0}t^2]\cup [2(1-2\alpha)\mu_0-C_{L_0},+\infty).
 \]
for sufficiently small $t$. 
Observe that the operators $\nabla_{t,e_i}$ and $\mathcal L_t$ depend smoothly on $t$ up to $t=0$. Their limits as $t\to 0$ are the operators
\[
\nabla_{0,e_i}=\nabla_{e_i}+\frac{1}{2}R^L_{x_0}(\mathcal R, e_i)
\]
and $\mathcal L_0$ given by \eqref{e:L0}. The spectrum of $\mathcal L_0$ consists of a discrete set of eigenvalues of infinite multiplicity (see, for instance, \cite[Theorem 1.15]{ma-ma08}). In particular, 
 \[
 \sigma(\mathcal L_0)\subset \{0\}\cup [2(1-2\alpha)\mu_0-C_{L_0},+\infty).
 \]

One can develop the rescaled operator $\mathcal L_t$ in a Taylor series in $t$. For the resulting asymptotic expansion, we refer the reader to \cite[Theorem 1.4]{ma-ma08}.

\section{Norm estimates of the resolvent}\label{norm}
The next step is the norm estimates. In this section, we derive the weighted norm estimates for the resolvent of the operator $\mathcal L_t$. 
First, we recall and slightly modify the results of \cite[Section 1.3]{ma-ma08}.

Denote by $C^\infty_b(\mathbb R^{2n},E_{x_0})$ the space of smooth functions on $\mathbb R^{2n}$ with values in $E_{x_0}$ whose derivatives of any order are uniformly bounded in $\mathbb R^{2n}$. So $a\in C^\infty_b(\mathbb R^{2n},E_{x_0})$ if, for any $\alpha\in \mathbb Z_+^{2n}$, we have  
\[
\sup_{Z\in \mathbb R^{2n}}|\nabla^{\alpha_1}_{e_1}\ldots \nabla^{\alpha_{2n}}_{e_{2n}} a(Z)| <\infty.
\]

For $m\in \mathbb N$ and $t>0$, let $\mathcal Q^m_t$ be the set of linear combinations of operators of the form $\nabla_{t,e_{i_1}}\ldots \nabla_{t,e_{i_j}}$, $j\leq m$, with coefficients from $C^\infty_b(\mathbb R^{2n},E_{x_0})$. It is easy to see that $\mathcal Q^m_t$ is independent of $t$, therefore, we will omit $t$ in the notation: $\mathcal Q^m_t=\mathcal Q^m$. Observe that, if $Q$ is in $\mathcal Q^m$, then the adjoint $Q^*$ is in $\mathcal Q^m$.

For $s\in C^\infty_c(\mathbb R^{2n},E_{x_0})$ set 
\[
\|s\|^2_{t,0}=\|s\|^2_{0}=\int_{\mathbb R^{2n}}|s(Z)|^2dv_{TX}(Z).
\]
and, for any $m\in \mathbb N$ and $t>0$,
\[
\|s\|^2_{t,m}=\sum_{\ell=0}^m\sum_{j_1,\ldots,j_\ell=1}^{2n}\|\nabla_{t,e_{j_1}}\cdots \nabla_{t,e_{j_\ell}}s\|^2_{t,0}
\]
Let $\langle\cdot,\cdot\rangle_{t,m}$ denote the inner product on $C^\infty(\mathbb R^{2n},E_{x_0})$ corresponding to $\|\cdot\|^2_{t,m}$. Let $H^m_t$ be the Sobolev space of order $m$ with norm $\|\cdot\|_{t,m}$. For any integer $m<0$, we define the Sobolev space $H^{m}_t$ by duality. It is easy to see that, for different $t_1$ and $t_2$, the norms $\|\cdot\|_{t_1,m}$ and $\|\cdot\|_{t_2,m}$ are equivalent, uniformly in $t_1, t_2\in [0,1]$.
For any bounded linear operator $A: H^m_t\to H^{m^\prime}_t$ with $m,m^\prime\in \mathbb Z$, we denote by $\|A\|^{m,m^\prime}_t$ its norm with respect to $\|\cdot\|_{t,m}$ and $\|\cdot\|_{t,m^\prime}$.   

Let $\delta$ be the circle in $\mathbb C$ oriented counterclockwise, centered at $0$ and of radius $c\mu_0$. The following theorem is a slight modification of \cite[Theorem 1.7]{ma-ma08}.

\begin{thm}\label{Thm1.7}
There exists $t_0>0$ such that the resolvent $(\lambda-\mathcal L_t)^{-1}$ exists for all $\lambda\in \delta$, $t\in [0,t_0]$. Moreover, there exists $C>0$ such that for all $\lambda\in \delta$, $t\in [0,t_0]$ and $x_0\in X$ we have 
\[
\|(\lambda-\mathcal L_t)^{-1}\|^{0,0}_t \leq C/\mu_0, \quad \|(\lambda-\mathcal L_t)^{-1}\|^{0,1}_t\leq C/\sqrt{\mu_0}, \quad \|(\lambda-\mathcal L_t)^{-1}\|^{-1,1}_t\leq C.
\]
\end{thm} 

\begin{proof}
The first inequality follows by the spectral theorem. 
Next, we have 
\begin{multline*}
\|(\lambda-\mathcal L_t)^{-1}s\|^2_{t,1} \leq \frac{1}{C_1}(\langle \mathcal L_t(\lambda-\mathcal L_t)^{-1}s,(\lambda-\mathcal L_t)^{-1}s \rangle_{t,0}\\ +C_2\|(\lambda-\mathcal L_t)^{-1}s\|^2_{t,0})
\leq \frac{C_2}{\mu_0}\|s\|^2_{t,0}.
\end{multline*}
For the third inequality, we refer to the proof of \cite[Theorem 1.7]{ma-ma08}.
\end{proof}

Now we introduce weighted spaces. Consider a function $f\in C^\infty(\mathbb R^{2n})$ given by 
\[
f(Z)=(1+|Z|^2)^{1/2}, \quad Z\in \mathbb R^{2n}. 
\] 
An important point is that $f$ satisfies the estimates
\[
C_1 |Z|\leq f(Z)\leq C_2|Z|, \quad |Z|\geq 1,
\]
with some $C_1, C_2>0$, and, for any $\alpha\in \mathbb Z_+^{2n}$ with $|\alpha|>0$,  
\begin{equation}\label{uniformD}
\sup_{Z\in \mathbb R^{2n}}|\nabla^{\alpha_1}_{e_1}\ldots \nabla^{\alpha_{2n}}_{e_{2n}} f(Z)| <\infty.
\end{equation}
For any $a\in \mathbb R$, let $L^2_a$ be the weighted $L^2$-space in $\mathbb R^{2n}$ with the weight $e^{af}$:
\[
L^2_a=\{s : e^{af}s\in L^2(\mathbb R^{2n},E_{x_0})\}.
\]
A family $\{\|\cdot\|_{L^2_{a,W}}: W\in \mathbb R^{2n}\}$ of equivalent norms in $L^2_a$ is defined by
\[
\|s\|^2_{L^2_{a,W}}=\|e^{af_W}s\|^2_{0}=\int_{\mathbb R^{2n}}e^{2af_W(Z)}|s(Z)|^2dv_{TX}(Z),
\]
where, for any $W\in \mathbb R^{2n}$, the function $f_W\in C^\infty(\mathbb R^{2n})$ is given by 
\[
f_W(Z)=f(Z-W)=(1+|Z-W|^2)^{1/2}, \quad Z\in \mathbb R^{2n}.
\] 
We will denote by $L^2_{a,W}$ the space $L^2_{a}$ with the norm $\|\cdot\|_{L^2_{a,W}}$. Similarly, one can introduce weighted Sobolev spaces.

Since the multiplication operator by $e^{af_W}$ defines a unitary operator $e^{af_W}: L^2_{a,W}\to L^2$, any bounded operator $A$ in $L^2_{a,W}$ is unitarily equivalent to the operator $A_{a,W}=e^{af_W}Ae^{-af_W}$ in $L^2$. In the sequel, instead of working directly with the weighted spaces $L^2_{a,W}$, we will consider operator families of the form $\{e^{af_W}Ae^{-af_W} : W\in \mathbb R^{2n}\}$ and only at the very end switch to weighted estimates.  

First, we observe that 
\[
\nabla_{t,a,W;e_j}:=e^{af_W}\nabla_{t,e_j}e^{-af_W}=\nabla_{t,e_j}-a\nabla_{e_j}f_W.
\]
In particular, this immediately implies that, if $Q$ in $\mathcal Q^m$, then the operator $e^{af_W}Qe^{-af_W}$ is in $\mathcal Q^m$. Moreover, the family $\{e^{af_W}Qe^{-af_W} : W\in \mathbb R^{2n}\}$ is a bounded family of operators belonging to $\mathcal Q^m$.

Next, for the operator  $\mathcal L_{t,a,W}:=e^{af_W}\mathcal L_te^{-af_W}$, we obtain 
\begin{align}\label{LtaZ}
 \mathcal L_{t,a,W}=& -g^{jk}(tZ)\left[\nabla_{t,a,W;e_j}\nabla_{t,a,W;e_k}- t\Gamma^{\ell}_{jk}(tZ)\nabla_{t,a,e_\ell;W}\right]-\tau(tZ) \\
\label{Lta}
= & \mathcal L_t+aA_W+a^2B_W, 
\end{align}
where 
\begin{align*}
A_W=& -\sum_{j,k=1}^{2n}g^{jk}(tZ)(\nabla_{e_j}f_W\nabla_{t,e_k}+\nabla_{e_k}f_W\nabla_{t,e_j}+\nabla_{e_j}\nabla_{e_k}f_W),\\
B_W=& -\sum_{j,k=1}^{2n}g^{jk}(tZ)\nabla_{e_j}f_W \nabla_{e_k}f_W.
\end{align*}
In particular, 
 \begin{equation}\label{ReLta}
{\mathrm Re}\, \mathcal L_{t,a,W}=\mathcal L_t - a^2\sum_{j,k=1}^{2n} g^{jk}(tZ)\nabla_{e_j} f_W \nabla_{e_k}f_W=\mathcal L_t - a^2\|df_W\|_{g^{-1}(tZ)}^2. 
\end{equation}

We have the following extension of \cite[Theorem 1.6]{ma-ma08}. 
\begin{thm}
There exist constants $C_1, C_2, C_3>0$ such that for any $t\in [0,1]$, $a\in \mathbb R$, $W\in \mathbb R^{2n}$ and $s, s^\prime\in C^\infty_c(\mathbb R^{2n})$ we have
\[
{\mathrm Re}\, \langle \mathcal L_{t,a,W} s, s\rangle_{t,0}\geq C_1\|s\|^2_{t,1}-(C_2+C^\prime_2a^2)\|s\|^2_{t,0},
\]
\[
\left|\langle \mathcal L_{t,a,W} s, s^\prime\rangle_{t,0}\right|\leq C_3(\|s\|_{t,1}\|s^\prime\|_{t,1}+a^2\|s\|_{t,0}\|s^\prime\|_{t,0}).
\]  
\end{thm}

\begin{proof} Using \eqref{uniformD}, \eqref{Lta}, \eqref{ReLta}, and \cite[Theorem 1.6]{ma-ma08}, we get 
\[
{\mathrm Re}\, \langle \mathcal L_{t,a,W} s, s\rangle_{t,0}\geq \langle \mathcal L_{t} s, s\rangle_{t,0}-C^\prime_2a^2\|s\|^2_{t,0}\geq C_1\|s\|^2_{t,1}-(C_2+C^\prime_2a^2)\|s\|^2_{t,0}
\]
and 
\begin{align*}
\left|\langle \mathcal L_{t,a,W} s, s^\prime\rangle_{t,0}\right|\leq & \left|\langle \mathcal L_{t} s, s^\prime\rangle_{t,0}\right| + |a| \left|\langle A_W s, s^\prime\rangle_{t,0}\right| + a^2 \left|\langle B_W s, s^\prime\rangle_{t,0}\right|\\ \leq & C_3(\|s\|_{t,1}\|s^\prime\|_{t,1}+a^2\|s\|_{t,0}\|s^\prime\|_{t,0}).\qedhere
\end{align*}
\end{proof}

Now we extend Theorem~\ref{Thm1.7} to the operators $\mathcal L_{t,a,W}$.

\begin{thm}\label{Thm1.7W}
There exist $c>0$ and $C>0$ such that, for all $\lambda\in \delta$, $t\in [0,t_0]$, $|a|<c\sqrt{\mu_0}$, $W\in \mathbb R^{2n}$ and $x_0\in X$ the inverse operator $(\lambda-\mathcal L_{t,a,W})^{-1}$ exists and
\[
\|(\lambda-\mathcal L_{t,a,W})^{-1}\|^{0,0}_t\leq C/\mu_0,  \quad \|(\lambda-\mathcal L_{t,a,W})^{-1}\|^{-1,1}_t\leq C.
\]
\end{thm}

\begin{proof}
By Theorem~\ref{Thm1.7}, it follows that,
for all $\lambda\in \delta$, $t\in [0,t_0]$, $a\in \mathbb R$, $W\in \mathbb R^{2n}$ and $x_0\in X$, we have
\begin{multline*}
\|( \mathcal L_{t,a,W}- \mathcal L_{t})(\lambda-\mathcal L_t)^{-1}\|^{0,0}_t = \|(aA_W+a^2B_W)(\lambda-\mathcal L_t)^{-1}\|^{0,0}_t\\ \leq C(a\|(\lambda-\mathcal L_t)^{-1}\|^{0,1}_t+a^2\|(\lambda-\mathcal L_t)^{-1}\|^{0,0}_t) \leq C\left(\frac{a}{\sqrt{\mu_0}}+\frac{a^2}{\mu_0}\right). 
\end{multline*}
Similarly, 
\begin{multline*}
\|( \mathcal L_{t,a,W}- \mathcal L_{t})(\lambda-\mathcal L_t)^{-1}\|^{-1,0}_t\\ \leq C(a\|(\lambda-\mathcal L_t)^{-1}\|^{-1,1}_t+a^2\|(\lambda-\mathcal L_t)^{-1}\|^{-1,0}_t)\leq C\left(a+\frac{a^2}{\sqrt{\mu_0}}\right).
\end{multline*}

Choose $c>0$ such that $C(c+c^2)<\frac 12$. If $|a|<c\sqrt{\mu_0}$, then the operator $\lambda-\mathcal L_{t,a,W}$ is invertible in $L^2$. Using the resolvent identity 
\[
(\lambda-\mathcal L_{t,a,W})^{-1}=(\lambda-\mathcal L_{t})^{-1}+(\lambda-\mathcal L_{t,a,W})^{-1}( \mathcal L_{t,a,W}- \mathcal L_{t})(\lambda-\mathcal L_t)^{-1},
\]
we infer
\begin{multline*}
\|(\lambda-\mathcal L_{t,a,W})^{-1}\|^{0,0}_t\\ 
\begin{aligned}
& \leq \|(\lambda-\mathcal L_{t})^{-1}\|^{0,0}_t+\|(\lambda-\mathcal L_{t,a,W})^{-1}\|^{0,0}_t\|( \mathcal L_{t,a,W}- \mathcal L_{t})(\lambda-\mathcal L_t)^{-1}\|^{0,0}_t\\ 
& \leq \|(\lambda-\mathcal L_{t})^{-1}\|^{0,0}_t+\frac 12\|(\lambda-\mathcal L_{t,a,W})^{-1}\|^{0,0}_t.
\end{aligned}
\end{multline*}
Therefore, 
\[
\|(\lambda-\mathcal L_{t,a,W})^{-1}\|^{0,0}_t\leq 2\|(\lambda-\mathcal L_{t})^{-1}\|^{0,0}_t\leq C/\mu_0.
\]
Similarly, 
\begin{multline*}
\|(\lambda-\mathcal L_{t,a,W})^{-1}\|^{0,1}_t \\
\begin{aligned}
& \leq \|(\lambda-\mathcal L_{t})^{-1}\|^{0,1}_t+\|(\lambda-\mathcal L_{t,a,W})^{-1}\|^{0,1}_t\|( \mathcal L_{t,a}- \mathcal L_{t})(\lambda-\mathcal L_t)^{-1}\|^{0,0}_t\\ 
& \leq \|(\lambda-\mathcal L_{t})^{-1}\|^{0,1}_t+\frac{1}{2}\|(\lambda-\mathcal L_{t,a,W})^{-1}\|^{0,1}_t.
\end{aligned}
\end{multline*}
Therefore, we get
\[
\|(\lambda-\mathcal L_{t,a,W})^{-1}\|^{0,1}_t\leq 
2\|(\lambda-\mathcal L_{t})^{-1}\|^{0,1}_t\leq C/\sqrt{\mu_0}.
\]

Finally, 
\begin{multline*}
\|(\lambda-\mathcal L_{t,a,W})^{-1}\|^{-1,1}_t \leq \|(\lambda-\mathcal L_{t})^{-1}\|^{-1,1}_t \\ +\|(\lambda-\mathcal L_{t,a,W})^{-1}\|^{0,1}_t\|( \mathcal L_{t,a,W}- \mathcal L_{t})(\lambda-\mathcal L_t)^{-1}\|^{-1,0}_t \leq C.\qedhere
\end{multline*}
\end{proof}

In the sequel, we will keep the notation $c$ for the constant given by Theorem \ref{Thm1.7W}, which will usually be related to the interval $(-c\sqrt{\mu_0}, c\sqrt{\mu_0})$ of admissible values of the parameter $a$.

Observe that, for any $\lambda\in \delta$, $t\in [0,t_0]$, $|a|<c\sqrt{\mu_0}$, $W\in \mathbb R^{2n}$ and $x_0\in X$, the operators $(\lambda-\mathcal L_{t,a,W})^{-1}$ and $(\lambda-\mathcal L_{t})^{-1}$ are related by the identity
\begin{equation}\label{LaL}
(\lambda-\mathcal L_{t,a,W})^{-1}=e^{af_W}(\lambda-\mathcal L_{t})^{-1}e^{-af_W},
\end{equation}
which should be understood in the following way. If $a<0$, then, for any $s\in C^\infty_c(\mathbb R^{2n},E_{x_0})$, the expression $e^{af_W}(\lambda-\mathcal L_{t})^{-1}e^{-af_W}s$ makes sense and defines a function in $L^2(\mathbb R^{2n},E_{x_0})$. Thus, we get a well defined operator 
\[
e^{af_W}(\lambda-\mathcal L_{t})^{-1}e^{-af_W}: C^\infty_c(\mathbb R^{2n},E_{x_0})\to L^2(\mathbb R^{2n},E_{x_0}),
\] 
one can check that $e^{af_W}(\lambda-\mathcal L_{t})^{-1}e^{-af_W}s=(\lambda-\mathcal L_{t,a,W})^{-1}s$ for any $s\in C^\infty_c(\mathbb R^{2n},E_{x_0})$. So \eqref{LaL} means that the operator $e^{af_W}(\lambda-\mathcal L_{t})^{-1}e^{-af_W}$ extends to a bounded operator in $L^2(\mathbb R^{2n},E_{x_0})$, which coincides with $(\lambda-\mathcal L_{t,a,W})^{-1}$. If $a>0$, then, for any $s\in L^2(\mathbb R^{2n},E_{x_0})$, the expression 
\[
e^{af_W}(\lambda-\mathcal L_{t})^{-1}e^{-af_W}s
\] 
makes sense as a distribution on $\mathbb R^{2n}$. Thus, we get a well defined operator 
\[
e^{af_W}(\lambda-\mathcal L_{t})^{-1}e^{-af_W}: L^2(\mathbb R^{2n},E_{x_0}) \to C^{-\infty}(\mathbb R^{2n},E_{x_0}).
\] 
So \eqref{LaL} means that this operator is indeed a bounded operator in $L^2(\mathbb R^{2n},E_{x_0})$, which coincides with $(\lambda-\mathcal L_{t,a,W})^{-1}$. 

The following proposition is an extension of Proposition 1.8 in \cite{ma-ma08}

\begin{prop}
For any natural $m$, there exists $C_m>0$ such that for any $t\in (0,1]$, $Q_1,\ldots, Q_m\in \{\nabla_{t,a,W;e_i}, Z_i\}_{i=1}^{2n}$, $|a|<c\sqrt{\mu_0}$, $W\in \mathbb R^{2n}$ and $s,s^\prime\in C^\infty_c(\mathbb R^{2n},E_{x_0})$ we have
\[
\left|\langle [Q_1, [Q_2,\ldots, [Q_m, \mathcal L_{t,a,W}]\ldots ]]s, s^\prime \rangle_{t,0}\right|\leq C_m\|s\|_{t,1}\|s^\prime\|_{t,1}
\]
\end{prop}

\begin{proof}
Recall the commutator relations 
\[
[\nabla_{t,e_i}, Z_j]=\delta_{ij}, \quad [\nabla_{t,e_i}, \nabla_{t,e_j}]=R^{L_0}(tZ)(e_i,e_j). 
\]
It follows that
\[
[\nabla_{t,a,W;e_i}, Z_j]=e^{af_W}[\nabla_{t,e_i}, Z_j]e^{-af_W}=\delta_{ij},
\] 
\[ 
[\nabla_{t,a,W;e_i}, \nabla_{t,a,W;e_j}]=e^{af_W}[\nabla_{t,e_i}, \nabla_{t,e_j}]e^{-af_W} =R^{L_0}(tZ)(e_i,e_j). 
\]
By \eqref{LtaZ}, the operator $\mathcal L_{t,a,W}$ has the form 
\[
\mathcal L_{t,a,W}=\sum_{i,j}a_{ij}(t,tZ)\nabla_{t,a,W;e_i}\nabla_{t,a,W;e_j}+\sum_{i}b_{i}(t,tZ)\nabla_{t,a,W;e_i}+c(t,tZ),
\]
where $a_{ij}(t,Z), b_{i}(t,Z), c(t,Z)$ as functions of $Z$ are in $C^\infty_b(\mathbb R^{2n},\operatorname{End}(E_{x_0}))$ with all the norms uniformly bounded on $t\in [0,1]$. Moreover, they are polynomials in $t$. Using the commutator relations, one can see that, for $Q_1,\ldots, Q_m\in \{\nabla_{t,a,W;e_i}, Z_i\}_{i=1}^{2n}$, the operator $[Q_1, [Q_2,\ldots, [Q_m, \mathcal L_{t,a,W}]\ldots ]]$ has the same structure as $\mathcal L_{t,a,W}$. If $(\nabla_{t,a,W;e_i})^*$ is the adjoint of $\nabla_{t,a,W;e_i}$ with respect to $\langle\cdot,\cdot\rangle_{t,0}$, then 
\[
(\nabla_{t,a,W;e_i})^*=-\nabla_{t,a,W;e_i}-t(\kappa^{-1}(e_i\kappa))(tZ)-2ae_i(f_W)(Z). 
\]
Using these facts, one can easily complete the proof. 
\end{proof}

The following result is an analog of Theorem 1.9 in \cite{ma-ma08}. For its proof, we can apply \emph{verbatim} the proof of that theorem.

\begin{thm}\label{Thm1.9}
For any $t\in (0,t_0]$, $\lambda\in \delta$, $m\in \mathbb N$, $W\in \mathbb R^{2n}$ and $|a|<c\sqrt{\mu_0}$, the resolvent $(\lambda-\mathcal L_{t,a,W})^{-1}$ maps $H^m_t$ to $H^{m+1}_t$. Moreover, for any $\alpha\in \mathbb Z_+^{2n}$ and $m\in \mathbb N$, there exists $C_{\alpha,m}>0$ such that for $t\in (0,t_0]$, $\lambda\in \delta$, $W\in \mathbb R^{2n}$ and $|a|<c\sqrt{\mu_0}$ we have
\[
\left\|Z^\alpha(\lambda-\mathcal L_{t,a,W})^{-1}s\right\|_{t,m+1}\leq C_{\alpha,m} \sum_{\alpha^\prime\leq \alpha}\left\|Z^{\alpha^\prime} s\right\|_{t,m}, \quad s\in C^\infty_c(\mathbb R^{2n},E_{x_0}).
\]
\end{thm}

\section{Norm estimates of the generalized Bergman projections}\label{norm-Bergman}
In this section, we derive the weighted norm estimates for the generalized Bergman projection associated with the operator $\mathcal L_t$ and its derivatives of an arbitrary order with respect to $t$.

We let the symbol $\mathcal P_{0,t}$ stand for the spectral projection for $\mathcal L_t$ corresponding to the interval $[-C_{L_0}t^2,C_{L_0}t^2]$. Let $\mathcal P_{q,t}(Z,Z^\prime)=\mathcal P_{q,t,x_0}(Z,Z^\prime)$ be the smooth kernel of $\mathcal P_{q,t}=(\mathcal L_t)^q\mathcal P_{0,t}$ with respect to $dv_{TX}$. 
For any integers $k>0$ and $q\geq 0$, we can write
\[
\mathcal P_{q,t}=(\mathcal L_t)^q\mathcal P_{0,t}=\frac{1}{2\pi i}
\begin{pmatrix}
q+k-1\\
k-1
\end{pmatrix}^{-1}\int_\delta \lambda^{q+k-1}(\lambda-\mathcal L_{t})^{-k}d\lambda. 
\]

\begin{prop}
For any $t\in (0,t_0]$, $Q, Q^\prime\in \mathcal Q^m$, $W\in \mathbb R^{2n}$ and $a, |a|<c\sqrt{\mu_0}$, the operator $Qe^{af_W}\mathcal P_{q,t}e^{-af_W}Q^\prime$ extends to a bounded operator in $L^2(\mathbb R^{2n},E_{x_0})$ with the norm bounded uniformly in $W$ and $t$.  
\end{prop}

\begin{proof}
First of all, we note that, for any $Q, Q^\prime\in \mathcal Q^m$, $W\in \mathbb R^{2n}$, and $a\in \mathbb R$, the operator $Qe^{af_W}\mathcal P_{q,t}e^{-af_W}Q^\prime $ is well defined as an operator from $C^\infty_c(\mathbb R^{2n},E_{x_0})$ to $C^{-\infty}(\mathbb R^{2n},E_{x_0})$.

By Theorem \ref{Thm1.9}, it follows that, for any $Q\in \mathcal Q^m$, there exists $C_{m}>0$ such that, for all $t\in (0,t_0]$, $\lambda\in\delta$, $W\in \mathbb R^{2n}$, and $|a|<c\sqrt{\mu_0}$ we have 
\[
\|Q(\lambda - \mathcal L_{t,a,W})^{-m}\|_t^{0,0}\leq C_m. 
\]
Since $\mathcal L_t$ is formally self-adjoint with respect to $\|\cdot\|_{t,0}$, we have $\mathcal L_{t,a,W}^*=\mathcal L_{t,-a,W}$, so after taking the adjoints, for all $t\in (0,t_0]$, $\lambda\in\delta$, $W\in \mathbb R^{2n}$ and $|a|<c\sqrt{\mu_0}$ we get  
\[
\|(\lambda - \mathcal L_{t,a,W})^{-m}Q\|_t^{0,0}\leq C_m. 
\]
Thus, for any $Q, Q^\prime\in \mathcal Q^m$, there exists $C_{m}>0$ such that, for all $t\in (0,t_0]$, $\lambda\in\delta$, $W\in \mathbb R^{2n}$ and $|a|<c\sqrt{\mu_0}$ we have
\[
\|Q(\lambda - \mathcal L_{t,a,W})^{-2m}Q^\prime\|_t^{0,0}\leq C_m. 
\]
By the above estimates, the desired statement follows immediately from the formula 
\begin{multline*}
Qe^{af_W} \mathcal P_{q,t}e^{-af_W}Q^\prime\\ 
\begin{aligned}
& =\frac{1}{2\pi i}
\begin{pmatrix}
q+k-1\\
k-1
\end{pmatrix}^{-1}\int_\delta \lambda^{q+k-1}Qe^{af_W} (\lambda-\mathcal L_{t})^{-k}e^{-af_W}Q^\prime d\lambda\\
& =\frac{1}{2\pi i}
\begin{pmatrix}
q+k-1\\
k-1
\end{pmatrix}^{-1}\int_\delta \lambda^{q+k-1}Q(\lambda-\mathcal L_{t,a,W})^{-k}Q^\prime d\lambda
\end{aligned} 
\end{multline*}
with $k>2m$, which can be justified much in the same way as relation \eqref{LaL} above.
\end{proof}

\begin{thm}
For any $r\geq 1$, $Q, Q^\prime\in \mathcal Q^m$, and $a, |a|<c\sqrt{\mu_0}$, there exists $C>0$ such that, for any $W\in \mathbb R^{2n}$ and $t\in (0,t_0]$, we have
\[
\|Qe^{af_W}\frac{\partial^r}{\partial t^r}\mathcal P_{q,t}e^{-af_W}Q^\prime s\|_{t,0}\leq C\sum_{|\beta|\leq 2r}\|Z^\beta s\|_{t,0}, \quad s\in C^\infty_c(\mathbb R^{2n},E_{x_0}).
\]
\end{thm}

\begin{proof}
We use the formula 
\[
e^{af_W}\frac{\partial^r}{\partial t^r}\mathcal P_{q,t}e^{-af_W}=\frac{1}{2\pi i}
\begin{pmatrix}
q+k-1\\
k-1
\end{pmatrix}^{-1}\int_\delta \lambda^{q+k-1}\frac{\partial^r}{\partial t^r}(\lambda-\mathcal L_{t,a,W})^{-k}d\lambda
\]
with $k>2(m+r+1)$. 

Put
\[
I_{k,r}=\left\{ ({\mathbf k}, {\mathbf r})=(k_0,\ldots,k_j,r_1,\ldots,r_j) : \sum_{i=0}^jk_i=k, \sum_{i=1}^jr_i=r, k_i,r_i\in \mathbb N \right\}. 
\]
Then we write
\begin{equation}\label{diff}
\frac{\partial^r}{\partial t^r}(\lambda-\mathcal L_{t,a,W})^{-k}=\sum_{({\mathbf k}, {\mathbf r})\in I_{k,r}} a^{\mathbf k}_{\mathbf r} A^{\mathbf k}_{\mathbf r}(\lambda,t,a,W),
\end{equation}
where the $a^{\mathbf k}_{\mathbf r}$ are some constants and 
\begin{multline*}
A^{\mathbf k}_{\mathbf r}(\lambda,t,a,W)\\ =(\lambda - \mathcal L_{t,a,W})^{-k_0} \frac{\partial^{r_1}\mathcal L_{t,a,W}}{\partial t^{r_1}}(\lambda - \mathcal L_{t,a,W})^{-k_1}\cdots  \frac{\partial^{r_j}\mathcal L_{t,a,W}}{\partial t^{r_j}}(\lambda - \mathcal L_{t,a,W})^{-k_j}. 
\end{multline*}
Now we can proceed as in the proof of \cite[Theorem 1.10]{ma-ma08}. We only observe that $\nabla_{t,a,W;e_j}=\nabla_{t,e_j}-a\nabla_{e_j}f_W$, and, for $r>0$, we have $\frac{\partial^r}{\partial t^r}\nabla_{t,a,W;e_j}=\frac{\partial^r}{\partial t^r}\nabla_{t,e_j}$. We deduce that for any $Q, Q^\prime\in \mathcal Q^m$, $W\in \mathbb R^{2n}$ and $a, |a|<c\sqrt{\mu_0}$, there exists $C>0$ such that 
\[
\|QA^{\mathbf k}_{\mathbf r}(\lambda,t,a,W)Q^\prime s\|_{t,0}\leq C\sum_{|\beta|\leq 2r}\|Z^\beta s\|_{t,0}, \quad s\in C^\infty_c(\mathbb R^{2n},E_{x_0}),
\]
for $\lambda\in \delta$ and $t\in (0,t_0]$. This completes the proof. 
\end{proof}

Using similar arguments, one can prove the following theorem. 

\begin{thm}\label{est-rem}
For any $r\geq 1$, $Q, Q^\prime\in \mathcal Q^m$,  $\alpha\in \mathbb Z_+^{2n}$ and $a, |a|<c\sqrt{\mu_0}$, there exists $C>0$ such that, for any $W\in \mathbb R^{2n}$ and $t\in (0,t_0]$, we have
\[
\|Q Z^\alpha e^{af_W}\frac{\partial^r}{\partial t^r}\mathcal P_{q,t}e^{-af_W} Q^\prime s\|_{t,0}\leq C\sum_{|\beta|\leq |\alpha|+2r}\| Z^\beta s\|_{t,0}, \quad s\in C^\infty_c(\mathbb R^{2n},E_{x_0}).
\] 
\end{thm}

\section{Asymptotic expansions and proofs of the main results} \label{asymp}
In this section, we complete the proof of the main theorem. First, recall that, by \cite[Theorem 1.11]{ma-ma08}, for any $s\in C^\infty_c(\mathbb R^{2n},E_{x_0})$, the limit of $\frac{\partial^r}{\partial t^r}\mathcal P_{q,t} s$ as $t\to 0$ in $L^2(\mathbb R^{2n},E_{x_0})$ exists and 
\[
\lim_{t\to 0}\frac{\partial^r}{\partial t^r}\mathcal P_{q,t} s =F_{q,r}s, 
\] 
where $F_{q,r}=F_{q,r,x_0}$ is the smoothing operator in $L^2(\mathbb R^{2n},E_{x_0})$ given by 
\[
F_{q,r}=\frac{1}{2\pi i}
\begin{pmatrix}
q+k-1\\
k-1
\end{pmatrix}^{-1}\int_\delta \lambda^{q+k-1}\sum_{({\mathbf k}, {\mathbf r})\in I_{k,r}} a^{\mathbf k}_{\mathbf r} A^{\mathbf k}_{\mathbf r}(\lambda,0) d\lambda
\]
with $k$ sufficiently large, and 
\[
A^{\mathbf k}_{\mathbf r}(\lambda,0)=(\lambda - \mathcal L_{0})^{-k_0} \frac{\partial^{r_1}\mathcal L_{t}}{\partial t^{r_1}}(0)(\lambda - \mathcal L_{0})^{-k_1}\cdots  \frac{\partial^{r_j}\mathcal L_{t}}{\partial t^{r_j}}(0)(\lambda - \mathcal L_{0})^{-k_j}.
\]

Observe that the estimates in Theorem~\ref{est-rem} are uniform in $t$ up to $t=0$, which immediately implies that the same statement is fulfilled for the limiting value $t=0$. We conclude that, for any $r\geq 1$, $Q, Q^\prime\in \mathcal Q^m$,  $\alpha\in \mathbb Z_+^{2n}$ and $a, |a|<c\sqrt{\mu_0}$, there exists $C>0$ such that, for any $W\in \mathbb R^{2n}$, we have 
\begin{multline}\label{est-rem1}
\|Q Z^\alpha e^{af_W}F_{q,r}e^{-af_W}Q^\prime s\|_{0,0}\\ \leq C\sum_{|\beta|\leq |\alpha|+2r}\| Z^\beta s\|_{0,0}, \quad s\in C^\infty_c(\mathbb R^{2n},E_{x_0}).
\end{multline} 

For any $q$ and $j$, put 
\[
R^{(j)}_{q,t}=P_{q,t}-\sum_{r=0}^jF_{q,r}t^r, \quad t>0.
\]

\begin{thm}
For any $Q, Q^\prime\in \mathcal Q^m$, $\alpha\in \mathbb Z_+^{2n}$  and $a, |a|<c\sqrt{\mu_0}$, there exists $C>0$ such that, for any $W\in \mathbb R^{2n}$ and $t\in [0,t_0]$, we have
\[
\|QZ^\alpha e^{af_W}R^{(j)}_{q,t}e^{-af_W} Q^\prime s\|_{t,0}\leq Ct^{j+1}\sum_{|\beta|\leq |\alpha|+2j+2}\|Z^\beta s\|_{t,0}, \quad s\in C^\infty_c(\mathbb R^{2n},E_{x_0}).
\]
\end{thm}

\begin{proof}
The statement follows immediately from the Taylor formula 
\[
\mathcal P_{q,t}-\sum_{r=0}^j\frac{1}{r!}F_{q,r} t^r=\frac{1}{j!}\int_0^t(t-\tau)^j\frac{\partial^{j+1}\mathcal P_{q,t}}{\partial t^{j+1}}(\tau) d\tau, \quad t\in [0,1],  
\]
and estimates \eqref{est-rem1}.
\end{proof}

\begin{thm}
For any $j,m,m^\prime\in \mathbb N$, $j\geq 2q$, there exist $C>0$ and $M>0$ such that for any $0\leq t\leq 1$ and $Z,Z^\prime\in T_{x_0}X$ 
\begin{multline}
\sup_{|\alpha|+|\alpha^\prime|\leq m}\Bigg|\frac{\partial^{|\alpha|+|\alpha^\prime|}}{\partial Z^\alpha\partial Z^{\prime\alpha^\prime}}\Bigg(\mathcal P_{q,t}(Z,Z^\prime)
-\sum_{r=2q}^jF_{q,r}(Z,Z^\prime)t^r\Bigg)\Bigg|_{\mathcal C^{m^\prime}(X)}\\ 
\leq Ct^{j+1}(1+|Z|+|Z^\prime|)^M\exp(-c\sqrt{\mu_0}|Z-Z^\prime|)+\mathcal O(t^{\infty}). 
\end{multline}
\end{thm}

\begin{proof}
For $M\in \mathbb N$, let $\mathcal D^M$ be the set of differential operators in $\mathbb R^{2n}$ of the form $\nabla_{e_{i_1}}\ldots \nabla_{e_{i_j}}$  with $j\leq M$. We claim that, for any $j\geq 2q$, $D, D^\prime\in \mathcal D^M$ and for any $a, |a|<c\sqrt{\mu_0}$ there exists $C>0$ such that 
\begin{equation}\label{DD}
\|De^{af_W}R^{(j)}_{q,t}e^{-af_W}D^\prime s\|_{t,0}\leq Ct^{j+1}\sum_{|\beta|\leq 2j+2M+2}\|Z^\beta s\|_{t,0}
\end{equation}
for any $0\leq t\leq 1$. To prove \eqref{DD}, we first observe that any $D\in \mathcal D^M$ can be written as 
\[
D=\sum_\alpha A_\alpha Q_\alpha,
\]
where $Q_\alpha \in \mathcal Q^M$ and the $A_\alpha\in C^\infty(\mathbb R^{2n},\operatorname{End}(E_{x_0}))$ satisfy the following condition: for any $\beta\in \mathbb Z_+^{2n}$, there exists $C_\beta>0$ such that
\begin{equation}\label{poly}
|\nabla^{\beta_1}_{e_1}\ldots \nabla^{\beta_{2n}}_{e_{2n}} A_\alpha(Z)| <C_\beta (1+|Z|)^M, \quad Z\in \mathbb R^{2n}.
\end{equation}
Similarly, any operator $D^\prime\in \mathcal D^M$ can be written as 
\[
D^\prime=\sum_{\alpha^\prime} Q^\prime_{\alpha^\prime} A^\prime_{\alpha^\prime},
\]
where $Q^\prime_{\alpha^\prime} \in \mathcal Q^M$ and the $A^\prime_{\alpha^\prime}\in C^\infty(\mathbb R^{2n},\operatorname{End}(E_{x_0}))$ satisfy \eqref{poly}. 

Then we have 
\[
\|De^{af_W}R^{(j)}_{q,t}e^{-af_W}D^\prime s\|_{t,0} \leq \sum_{\alpha,\alpha^\prime} \|A_\alpha Q_\alpha e^{af_W}R^{(j)}_{q,t}e^{-af_W}Q^\prime_{\alpha^\prime} A^\prime_{\alpha^\prime} s\|_{t,0}.
\]
For every term in the right-hand side of the last inequality, we get
\begin{multline*}
\|A_\alpha Q_\alpha e^{af_W}R^{(j)}_{q,t}e^{-af_W}Q^\prime_{\alpha^\prime} A^\prime_{\alpha^\prime} s\|_{t,0}\\
\begin{aligned}
\leq & C\|(1+|Z|)^m Q_\alpha e^{af_W}R^{(j)}_{q,t}e^{-af_W}Q^\prime_{\alpha^\prime} A^\prime_{\alpha^\prime} s\|_{t,0}\\ \leq & Ct^{j+1}\sum_{|\beta|\leq 2j+M+2}\| Z^\beta A^\prime_{\alpha^\prime} s\|_{t,0}\\ \leq & Ct^{j+1}\sum_{|\beta|\leq 2j+2M+2}\| Z^\beta s\|_{t,0},
\end{aligned}
\end{multline*}
which completes the proof of \eqref{DD}. 

Let $H^m(\mathbb R^{2n},E_{x_0})$ denote the usual Sobolev space on $\mathbb R^{2n}$ with the norm 
\[
\|u\|_{H^m(\mathbb R^{2n},E_{x_0})}=\left(\int_{\mathbb R^{2n}}(1+|\xi|^2)^{m/2}|\hat u(\xi)|^2\right)^{1/2},
\]
$\hat u$ is the Fourier transform of $u$. By \eqref{DD}, it follows that, for any $m ,m^\prime \in \mathbb R$ and $a, |a|<c\sqrt{\mu_0}$ there exist $M\in\mathbb N$ and $C>0$ such that 
\[
\|e^{af_W}R^{(j)}_{q,t}e^{-af_W}s\|_{H^m(\mathbb R^{2n},E_{x_0})}\leq Ct^{j+1}\sum_{|\beta|\leq M}\|Z^\beta s\|_{H^{m^\prime}(\mathbb R^{2n},E_{x_0})}
\]
for any $0<t\leq 1$ and $s\in C^\infty_c(\mathbb R^{2n},E_{x_0})$. In particular, for any $s\in H^{m^\prime}(\mathbb R^{2n},E_{x_0})$ with $\operatorname{supp} s\subset B(0,\sigma)$ for some $\sigma>0$, we have 
\begin{equation}\label{delta-sigma}
\|e^{af_W}R^{(j)}_{q,t}e^{-af_W} s\|_{H^m(\mathbb R^{2n},E_{x_0})}\leq Ct^{j+1}(1+\sigma)^{M}\|s\|_{H^{m^\prime}(\mathbb R^{2n},E_{x_0})}.
\end{equation}

For any $v\in E_{x_0}$, consider the delta-function $\delta^v_{0}\in C^{-\infty}(\mathbb R^{2n},E_{x_0})$ defined by $\langle \delta^v_{0},s\rangle=\langle v, s(0)\rangle_{h^{E_{x_0}}}$ for $s\in C^{\infty}_c(\mathbb R^{2n},E_{x_0})$. Let $\delta^v_{Z^\prime}\in C^{-\infty}(\mathbb R^{2n},E_{x_0})$ be the delta-function at $Z^\prime\in \mathbb R^{2n}$: $\delta^v_{Z^\prime}(Z)=\delta^v_0(Z-Z^\prime)$. Then $\delta^v_{Z^\prime}\in H^{-(n+1)}(\mathbb R^{2n},E_{x_0})$, with the norm, uniformly bounded in $Z^\prime$ and $v$ with $|v|_{h^{E_{x_0}}}=1$ (actually, independent of $Z^\prime$). We can write 
\[
\frac{\partial^{|\alpha^\prime|}}{\partial Z^{\prime\alpha^\prime}}\left(e^{af_W(Z)}R^{(j)}_{q,t}(Z,Z^\prime)e^{-af_W(Z^\prime)}v\right)=\left(e^{af_W}R^{(j)}_{q,t}e^{-af_W}D^\prime_{\alpha^\prime}\delta^v_{Z^\prime}\right)(Z)
\]
with some $D^\prime_{\alpha^\prime}\in \mathcal D^m$. Therefore, we get 
\begin{multline*}
\sup_{|\alpha|+|\alpha^\prime|\leq m}\left|\frac{\partial^{|\alpha|+|\alpha^\prime|}}{\partial Z^\alpha\partial Z^{\prime\alpha^\prime}}\left(e^{af_W(Z)}R^{(j)}_{q,t}(Z,Z^\prime)e^{-af_W(Z^\prime)}v\right)\right|\\
\leq C\sup_{|\alpha^\prime|\leq m}\left\|\frac{\partial^{|\alpha^\prime|}}{\partial Z^{\prime\alpha^\prime}}\left(e^{af_W(Z)}R^{(j)}_{q,t}(Z,Z^\prime)e^{-af_W(Z^\prime)}v\right)\right\|_{C^m_b(\mathbb R^{2n},E_{x_0})}\\
= C\sum_{\alpha^\prime}\left\|e^{af_W}R^{(j)}_{q,t}e^{-af_W}D^\prime_{\alpha^\prime}\delta^v_{Z^\prime}\right\|_{C^m_b(\mathbb R^{2n},E_{x_0})}.
\end{multline*}

By the Sobolev embedding theorem, $H^M(\mathbb R^{2n},E_{x_0})\hookrightarrow C^m_b(\mathbb R^{2n},E_{x_0})$ for $M=m+n+1$. Therefore, using \eqref{delta-sigma}, for $|Z^\prime|\leq \sigma$ with an arbitrary $\sigma>0$ we get
\begin{multline*}
\left\|e^{af_W}R^{(j)}_{q,t}e^{-af_W}D^\prime_{\alpha^\prime}\delta^v_{Z^\prime}\right\|_{C^m_b(\mathbb R^{2n},E_{x_0})}\\
\begin{aligned}
 & \leq C\left\|e^{af_W}R^{(j)}_{q,t}e^{-af_W}D^\prime_{\alpha^\prime}\delta^v_{Z^\prime}\right\|_{H^{m+n+1}(\mathbb R^{2n},E_{x_0})}\\ & \leq Ct^{j+1}(1+\sigma)^{2j+2m+2n+4}\|\delta^v_{Z^\prime}\|_{H^{-(m+n+1)}(\mathbb R^{2n},E_{x_0})}.
 \end{aligned}
\end{multline*}
Setting $W=Z^\prime$, we obtain
\[
\sup_{|\alpha|+|\alpha^\prime|\leq m}\left|\frac{\partial^{|\alpha|+|\alpha^\prime|}}{\partial Z^\alpha\partial Z^{\prime\alpha^\prime}}\left(e^{af_{Z^\prime}(Z)}R^{(j)}_{q,t}(Z,Z^\prime)\right)\right|
\leq Ct^{j+1}(1+\sigma)^{2j+2m+2n+4} 
\]
and 
\begin{align*}
\sup_{|\alpha|+|\alpha^\prime|\leq m}\left|\frac{\partial^{|\alpha|+|\alpha^\prime|}}{\partial Z^\alpha\partial Z^{\prime\alpha^\prime}}R^{(j)}_{q,t}(Z,Z^\prime)\right|
& \leq Ct^{j+1}(1+\sigma)^{2j+2m+2n+4}e^{-af_{Z^\prime}(Z)} \\
& \leq Ct^{j+1}(1+|Z^\prime|)^{2j+2m+2n+4}e^{-a|Z-Z^\prime|},
\end{align*}
which completes the proof in the case $m^\prime=0$. 

To treat the case where $m^\prime \geq 1$, we proceed as in the proof of \cite[Theorem 1.10]{ma-ma08}. For any vector $U\in T_{x_0}X$, the following formula holds:
\[
\nabla_U\mathcal P_{q,t}=(\mathcal L_t)^q\mathcal P_{0,t}=\frac{1}{2\pi i}
\begin{pmatrix}
q+k-1\\
k-1
\end{pmatrix}^{-1}\int_\delta \lambda^{q+k-1}\nabla_U (\lambda-\mathcal L_{t})^{-k}d\lambda. 
\]
The operator $\nabla_U (\lambda-\mathcal L_{t})^{-k}$ is given by a formula similar to \eqref{diff}. Then we observe that $\nabla_U \mathcal L_{t}$ is a differential operator on $T_{x_0}X$ with the same structure as $\mathcal L_t$. This allows us to extend all our considerations to the case of an arbitrary $m^\prime  \geq 1$. 
\end{proof}

To complete the proof of Theorem \ref{t:main}, we observe that, by \eqref{scaling},
\[
P^0_{q,p}(Z,Z^\prime)=t^{-2n-2q}\kappa^{-\frac 12}(Z)\mathcal P_{q,t}(Z/t,Z^\prime/t)\kappa^{-\frac 12}(Z^\prime), \quad Z,Z^\prime \in \mathbb R^{2n},
\]
and make use of Proposition \ref{prop13}. We remark that, by \cite[Theorem 1.18]{ma-ma08}, we have $F_{q,r}=0$ for $q>0$, $r<2q$. 

\section{Toeplitz operators}\label{Toeplitz}

In this section, we construct the algebra of Toeplitz operators associated with the renormalized Bochner Laplacian on the symplectic manifold $X$. The proofs of the results of this section are obtained by a word for word repetition of the arguments of the paper \cite{ma-ma08a}. So we only give basic definitions and statements of the main results. As mentioned in the Introduction, the algebra of Toeplitz operators associated with the renormalized Bochner Laplacian was also constructed in \cite{ioos-lu-ma-ma}. 
  
\begin{defn}
A Toeplitz operator is a sequence $\{T_p\}=\{T_p\}_{p\in \mathbb N}$ of bounded linear operators $T_p : L^2(X,L^p\otimes E)\to L^2(X,L^p\otimes E)$, satisfying the following conditions.
\begin{description}
\item[(i)] For any $p\in \mathbb N$, we have 
\[
T_p=P_{\mathcal H_p}T_pP_{\mathcal H_p}. 
\]
\item[(ii)] There exists a sequence $g_l\in C^\infty(X,\operatorname{End}(E))$ such that 
\[
T_p=P_{\mathcal H_p}\left(\sum_{l=0}^\infty p^{-l}g_l\right)P_{\mathcal H_p}+\mathcal O(p^{-\infty}),
\]
i.e. for any natural $k$ there exists $C_k>0$ such that 
\[
\left\|T_p-P_{\mathcal H_p}\left(\sum_{l=0}^k p^{-l}g_l\right)P_{\mathcal H_p}\right\|\leq C_kp^{-k-1}.
\]
\end{description}
  \end{defn}
  
The full symbol of $\{T_p\}$ is the formal series $\sum_{l=0}^\infty \hbar^lg_l\in C^\infty(X,\operatorname{End}(E))[[\hbar ]]$ and the principal symbol of $\{T_p\}$ is $g_0$.  

In the particular case when $g_l=0$ for $l\geq 1$ and $g_0=f$, we get the operator $T_{f,p}=P_{\mathcal H_p}fP_{\mathcal H_p}: L^2(X,L^p\otimes E)\to L^2(X,L^p\otimes E)$. The Schwartz kernel of $T_{f,p}$ is given by 
\[
T_{f,p}(x,x^\prime)=\int_X P_p(x,x^{\prime\prime})f(x^{\prime\prime})P_p(x^{\prime\prime},x^{\prime})dv_X(x^{\prime\prime}). 
\] 

\begin{lem}
For any $\varepsilon>0$ and $l,m\in \mathbb N$, there exists $C>0$ such that for any $p\geq 1$ and $(x,x^\prime)\in X\times X$ with $d(x,x^\prime)>\varepsilon$ we have
\[
|T_{f,p}(x,x^\prime)|_{C^m}\leq Cp^{-l}. 
\]
\end{lem}

Let $\{\Xi_p\}$ be a sequence  of linear operators $\Xi_p : L^2(X,L^p\otimes E)\to L^2(X,L^p\otimes E)$ with smooth kernel $\Xi_p(x,x^\prime)$ with respect to $dv_X$.  As described in the Introduction, $\Xi_p(x,x^\prime)$ induces a smooth section $\Xi_{p,x_0}(Z,Z^\prime)$ of the vector bundle $\pi^*(\operatorname{End}(E))$ on $TX\times_X TX$ defined for all $x_0\in X$ and $Z,Z^\prime\in T_{x_0}X$ with $|Z|, |Z^\prime|<a_X$. Recall that $\mathcal P=\mathcal P_{x_0}$ denotes the Bergman kernel in $\mathbb R^{2n}$ given by \eqref{e:Bergman}. 

\begin{defn}
We say that 
\[
p^{-n}\Xi_{p,x_0}(Z,Z^\prime)\cong \sum_{r=0}^k(Q_{r,x_0}\mathcal P_{x_0})(\sqrt{p}Z,\sqrt{p}Z^\prime)p^{-\frac{r}{2}}+\mathcal O(p^{-\frac{k+1}{2}})
\]
with some $Q_{r,x_0}\in \operatorname{End}(E_{x_0})[Z,Z^\prime]$, $0\leq r\leq k$,  depending smoothly on $x_0\in X$, if there exist $\varepsilon^\prime\in (0,a_X]$ and $C_0>0$ with the following property:
for any $l\in \mathbb N$, there exist $C>0$ and $M>0$ such that for any $x_0\in X$, $p\geq 1$ and $Z,Z^\prime\in T_{x_0}X$, $|Z|, |Z^\prime|<\varepsilon^\prime$, we have 
\begin{multline*}
\Bigg|p^{-n}\Xi_{p,x_0}(Z,Z^\prime)\kappa^{\frac 12}(Z)\kappa^{\frac 12}(Z^\prime) -\sum_{r=0}^k(Q_{r,x_0}\mathcal P_{x_0})(\sqrt{p} Z, \sqrt{p}Z^\prime)p^{-\frac{r}{2}}\Bigg|_{\mathcal C^{l}(X)}\\ 
\leq Cp^{-\frac{k+1}{2}}(1+\sqrt{p}|Z|+\sqrt{p}|Z^\prime|)^M\exp(-\sqrt{C_0p}|Z-Z^\prime|)+\mathcal O(p^{-\infty}).
\end{multline*}
\end{defn}

By Theorem \ref{t:main}, for any $k\in \mathbb N$, we have
\[
p^{-n}P_{p,x_0}(Z,Z^\prime)\cong
\sum_{r=0}^k(F_{0,r,x_0}\mathcal P_{x_0})(\sqrt{p} Z, \sqrt{p}Z^\prime)p^{-\frac{r}{2}}+\mathcal O(p^{-\frac{k+1}{2}}).
\]

For any polynomial $F\in \mathbb C[Z,Z^\prime]$, consider the operator $F\mathcal P$ in $L^2(T_{x_0}X)\cong L^2(\mathbb R^{2n})$ with the kernel $(F\mathcal P)(Z,Z^\prime)$ with respect to $dZ$.
For any polynomials $F,G\in \mathbb C[Z,Z^\prime]$, define the polynomial $\mathcal K[F,G]\in \mathbb C[Z,Z^\prime]$ by the condition
\[
((F\mathcal P)\circ (G\mathcal P))(Z,Z^\prime)=(\mathcal K[F,G]\mathcal P)(Z,Z^\prime), 
\]
where $(F\mathcal P)\circ (G\mathcal P)$ is the composition of the operators $F\mathcal P$ and $G\mathcal P$ in $L^2(T_{x_0}X)$. 

\begin{lem}
Let $f\in C^\infty(X, \operatorname{End}(E))$. For any $k\in \mathbb N$, $x_0\in X$, $Z,Z^\prime\in T_{x_0}X$, $|Z|, |Z^\prime|<\varepsilon/2$,  we have 
\[
p^{-n}T_{f,p,x_0}(Z,Z^\prime)\cong \sum_{r=0}^k(Q_{r,x_0}(f)\mathcal P_{x_0})(\sqrt{p}Z,\sqrt{p}Z^\prime)p^{-\frac{r}{2}}+\mathcal O(p^{-\frac{k+1}{2}}),
\]
where the polynomials $Q_{r,x_0}(f)\in \operatorname{End}(E_{x_0})[Z,Z^\prime]$ have the same parity as $r$ and are given by 
\[
Q_{r,x_0}(f)=\sum_{r_1+r_2+|\alpha|=r}\mathcal K\left[F_{0,r_1,x_0}, \frac{\partial^\alpha f_{x_0}}{\partial Z^\alpha}(0)\frac{Z^\alpha}{\alpha!}F_{0,r_2,x_0}\right],
\]
In particular, 
\begin{align*}
Q_{0,x_0}(f)&=f(x_0),\\
Q_{1,x_0}(f)&=f(x_0)F_{0,1,x_0}+\mathcal K\left[F_{0,0,x_0}, \frac{\partial f_{x_0}}{\partial Z_j}(0)Z_jF_{0,0,x_0}\right].
\end{align*}
\end{lem} 
 
We have the following criterion for Toeplitz operators. 

\begin{thm}
A family $\{T_p: L^2(X,L^p\otimes E)\to L^2(X,L^p\otimes E)\}$ of bounded linear operators is a Toeplitz operator if and only if it satisfies the following three conditions:
\begin{description}
\item[(i)] For any $p\in \mathbb N$, we have
\[
T_p=P_{\mathcal H_p}T_pP_{\mathcal H_p}. 
\]
\item[(ii)] For any $\varepsilon_0>0$ and $l\in \mathbb N$, there exists $C>0$ such that 
\[
|T_{f,p}(x,x^\prime)|\leq Cp^{-l}
\]
for any $p\geq 1$ and $(x,x^\prime)\in X\times X$ with $d(x,x^\prime)>\varepsilon_0$.
\item[(iii)] There exist a family of polynomials $\mathcal Q_{r,x_0}\in \operatorname{End}(E_{x_0})[Z,Z^\prime]$, depending smoothly on $x_0$, of the same parity as $r$ and $\varepsilon^\prime\in (0,a_X/4)$ such that
\[
p^{-n}T_{p,x_0}(Z,Z^\prime)\cong \sum_{r=0}^k(\mathcal Q_{r,x_0}\mathcal P_{x_0})(\sqrt{p}Z,\sqrt{p}Z^\prime)p^{-\frac{r}{2}}+\mathcal O(p^{-\frac{k+1}{2}})
\] 
for any $k\in \mathbb N$, $x_0\in X$, $Z,Z^\prime\in T_{x_0}X$, $|Z|, |Z^\prime|<\varepsilon^\prime$.  
\end{description}
\end{thm} 

Using this criterion, one can show that the set of Toeplitz operators is an algebra. 

\begin{thm}
Let $f,g\in C^\infty(X,\operatorname{End}(E))$. Then the composition of the Toeplitz operators $T_{f,p}$ and $T_{g,p}$ is a Toeplitz operator. More precisely, it admits the asymptotic expansion 
\[
T_{f,p}T_{g,p}=\sum_{r=0}^\infty p^{-r}T_{C_r(f,g),p}+\mathcal O(p^{-\infty}), 
\]
with some $C_r(f,g)\in C^\infty(X,\operatorname{End}(E))$, where the $C_r$ are bidifferential operators. In particular, $C_0(f,g)=fg$ and, for $f,g\in C^\infty(X)$, we have
\[
C_1(f,g)-C_1(f,g)=i\{f,g\},
\]
where $\{f,g\}$ is the Poisson bracket on $(X,2\pi \omega)$.
\end{thm}

\end{document}